\documentclass[letterpaper,10pt]{article}

\usepackage[letterpaper,left=1in,right=1in,top=1.5in,bottom=1.5in]{geometry}

\title{Weibel's conjecture for twisted K-theory}
\author{Joel Stapleton}
\date{\today}

\usepackage{amsfonts}
\usepackage{amsmath}
\usepackage{amssymb}
\usepackage{amsthm}
\usepackage{cite}
\usepackage{extarrows}
\usepackage{extramarks}
\usepackage{fancyhdr}
\usepackage{graphicx}
\usepackage{hyperref}
\hypersetup{
        colorlinks=true,
        linkcolor=cyan,
        citecolor=blue,
}
\usepackage{mathrsfs}
\usepackage{scalerel}
\usepackage{tikz}
\usepackage{tikz-cd}
\usepackage{url}
\usepackage{wrapfig}
\usepackage{lipsum}

\graphicspath{{figures/}}

\hypersetup{
    colorlinks=true,
    linkcolor=blue,
    filecolor=magenta,      
    urlcolor=blue,
}

\newtheorem{theorem}{Theorem}[section]
\newtheorem{lemma}[theorem]{Lemma}

\newtheorem{proposition}[theorem]{Proposition}
\newtheorem{corollary}[theorem]{Corollary}

\newcommand{\ZZ}{\mathbb{Z}}

\newcommand{\Einf}{\mathbb{E}_\infty}
\newcommand{\Eone}{\mathbb{E}_1}

\newcommand{\mm}{\mathfrak{m}}

\newcommand{\aff}{\mathbb{A}}    

\newcommand{\proj}{\mathbb{P}}

\newcommand{\cat}[1]{\tn{\textbf{#1}}}

\newcommand{\Mod}{\textnormal{Mod}}

\newcommand{\Perf}{\textnormal{Perf}}

\newcommand{\PrLstomega}{\textnormal{Pr}^{\textnormal{L}}_{\textnormal{st}, \omega}}
\newcommand{\Catperf}{\textnormal{Cat}^{\textnormal{perf}}_{\infty}}

\newcommand{\Spectra}{\textnormal{Sp}}

\newcommand{\cofib}{\textnormal{cofib}}
\newcommand{\fib}{\textnormal{fib}}

\newcommand{\etale}{\'etale}

\newcommand{\torsheaf}{\mathscr{T}or}

\newcommand{\OO}{\mathscr{O}}

\newcommand{\Spec}{\textnormal{Spec}\,}

\newcommand{\coker}{\textnormal{coker}\,}

\newcommand{\surj}{\twoheadrightarrow}
\newcommand{\inj}{\hookrightarrow}

\newcommand{\tn}[1]{\textnormal{#1}}

\newcommand{\hmtpy}{\simeq}

\newcommand{\sheaffont}[1]{\mathcal{#1}}

\newcommand{\pitor}{\pi_1} 
\newcommand{\piaff}{\pi_2} 

\newcommand{\low}{1} 

\newcommand{\azu}{\sheaffont{A}}

\newtheorem{innercustomgeneric}{\customgenericname}
\providecommand{\customgenericname}{}
\newcommand{\newcustomtheorem}[2]{%
  \newenvironment{#1}[1]
  {%
   \renewcommand\customgenericname{#2}%
   \renewcommand\theinnercustomgeneric{##1}%
   \innercustomgeneric
  }
  {\endinnercustomgeneric}
}
\theoremstyle{definition}

\theoremstyle{definition}
\newtheorem{rem}[theorem]{Remark}
\theoremstyle{definition}
\newtheorem{defin}[theorem]{Definition}

\newcustomtheorem{customthm}{Theorem}
\newcustomtheorem{customprop}{Proposition}
\newcustomtheorem{customexample}{Example}

\makeatletter
\newcommand{\colim@}[2]{%
  \vtop{\m@th\ialign{##\cr
    \hfil$#1\operator@font colim$\hfil\cr
    \noalign{\nointerlineskip\kern1.5\ex@}#2\cr
    \noalign{\nointerlineskip\kern-\ex@}\cr}}%
}
\newcommand{\colim}{%
  \mathop{\mathpalette\colim@{\rightarrowfill@\scriptscriptstyle}}\nmlimits@
}
\makeatother

\begin{document}
\maketitle

\begin{abstract}
  \noindent We prove Weibel's conjecture for twisted $K$-theory when twisting by a smooth proper
  connective dg-algebra. Our main contribution is showing we can kill a negative twisted $K$-theory
  class using a projective birational morphism (in the same twisted setting). We extend the vanishing result to relative twisted $K$-theory of a smooth affine
  morphism and describe counter examples to some similar extensions.
\end{abstract}
\section{Introduction}

 The so-called fundamental theorem for $K_1$ and $K_0$ states that for any ring $R$ there
is an exact sequence
\[
  0 \rightarrow K_1(R) \rightarrow K_1(R[t]) \oplus K_1(R[t^{-1}]) \rightarrow K_1(R[t^\pm])
  \rightarrow K_0(R) \rightarrow 0.
\]
We see $K_0$ can be defined using $K_1$. There is an analogous exact sequence, truncated on the right,
for $K_0$. Bass defines $K_{-1}(X)$ as the cokernel of the final morphism. He then iterates the construction to
define a theory of negative K-groups \cite[Sections XII.7 and XII.8]{bass_algebraic_k-theory}.

Weibel's conjecture, originally posed in \cite{weibel_conjecture}, asks if $K_{-i}(R) = 0$
for $i > \dim R$ when $R$ has finite Krull dimension. Kerz--Strunk--Tamme \cite{kerz_strunk_tamme} have proven Weibel's conjecture
for any Noetherian scheme of finite Krull dimension (see the introduction for a historical summary of progress) by establishing pro
cdh-descent for algebraic $K$-theory. 
Land--Tamme \cite{land_tamme} have shown that a general class of localizing invariants
satisfy pro cdh-descent. With this improvement, we extend Weibel's vanishing to some cases of twisted
$K$-theory. 

\begin{theorem}\label{main_theorem}
  Let $X$ be a Noetherian $d$-dimensional scheme and
  $\azu$ a sheaf of smooth proper connective quasi-coherent differential graded algebras over $X$, then
  $K_{-i}(\Perf(\azu))$ vanishes for $i > d$.
\end{theorem}
The original goal of this paper was to extend Weibel's conjecture to an Azumaya algebra over a
scheme. To an Azumaya algebra $\azu$ of rank $r^2$ on $X$ we can associate a Severi-Brauer variety
$P$ of relative dimension $r-1$ over
$X$. Such a variety is \etale-locally isomorphic over $X$ to $\proj^{r-1}_X$. In Quillen's work
\cite{quillen_higher_algebraic_k-theory}, he generalizes the projective bundle formula to
Severi-Brauer varities showing (for $i \geq 0$)
\[
  K_i(P) \cong \bigoplus_{n=0}^{r-1} K_i(\azu^{\otimes n}).
\]
At the root of this computation is a semi-orthogonal decomposition of $\Perf(P)$. Consequently, the computation lifts to the level of nonconnective
$K$-theory spectra. Statements about the $K$-theory of
Azumaya algebras can generally be extracted through this decomposition. In our case, the dimension
of the Severi-Brauer variety jumps and so Weibel's conjecture (for our noncommutative dg-algebra) does not follow from the commutative setting.

We could remedy this by characterizing a class of morphisms to $X$, which should include Severi-Brauer
varieties, and then show the relative $K$-theory vanishes under $-d-1$. In Remark
\ref{counter_example}, we show that smooth and proper morphisms (in fact, smooth and projective) are
not sufficient. We warn the reader that we will use the overloaded words ``smooth and proper'' in
both the scheme and dg-algebra settings.

For dg-algebras and dg-categories, properness and smoothness are module and algebraic finiteness
conditions, see To\"en--Vaqui\'e \cite[Definition 2.4]{toen_vaquie}. Together, the two conditions
characterize the dualizable objects in $\Mod_{\Mod_R}(\PrLstomega)$, whose objects are
$\omega$-compactly generated $R$-linear stable presentable $\infty$-categories. More surprisingly,
the invertible objects of $\Mod_{\Mod_R}(\PrLstomega)$ are
exactly the module categories over derived Azumaya algebras, see Antieau--Gepner \cite[Theorem 3.15]{antieau_gepner}. So Theorem \ref{main_theorem} recovers the discrete Azumaya algebra case.

However, any connective derived Azumaya algebras is discrete. After base-changing to a field $k$, $\azu_k
\cong H_*\azu_k$ is
a connective graded $k$-algebra and $H_*\azu_k \otimes_k (H_*\azu_k)^{op}$ is Morita equivalent to
$k$. So $H_*\azu_k$ is discrete. The scope of Theorem \ref{main_theorem} is not wasted
as smooth proper connective
dg-algebras can be nondiscrete, see Raedschelders--Stevenson \cite[Section 4.3]{raedschelders_stevenson}.

The proof of Theorem \ref{main_theorem} follows Kerz \cite{kerz}. In Section $2$, we define and study twisted
$K$-theory. We kill a
negative twisted $K$-theory class using a projective birational morphism in Section $3$. Lastly, Section
$4$ holds the main theorem and we consider some extensions.

\textbf{Conventions:}
We make very little use of the language of $\infty$-categories. For a commutative ring $R$, there is
an equivalence of $\infty$-categories between the $\Eone$-ring spectra over $HR$ and differential
graded algebras over $R$ localized at the quasi-isomorphisms (see \cite[7.1.4.6]{lurie_ha}). For a
dg-algebra (or $\Eone$-ring) $\azu$, we can consider the $\infty$-category $\mathrm{RMod}(\azu)$ of spectra which
have a right $\azu$-module structure. We will refer to this $\infty$-category as the derived category of
$\azu$ and denote it by $D(\azu)$. The subcategory $\Perf(\azu)$ consists of all compact objects of
$\mathrm{RMod}(\azu)$, or the right $\azu$-modules which corepresent a functor that commutes with
filtered colimits. We shall refer to objects of $\Perf(\azu)$ as perfect complexes over $\azu$.

We use $K(-)$ undecorated as non-connective algebraic $K$-theory and consider it
as a localizing invariant in the sense of Blumberg--Gepner--Tabuada \cite{blumberg_gepner_tabuada_universal}. In
particular, it is an $\infty$-functor from $\Catperf$, the $\infty$-category of idempotent complete small stable infinity categories
with exact functors, taking values in $\Spectra$, the $\infty$-category of spectra. For $X$ a quasi-compact quasi-separated
scheme, $K(\Perf(X))$ is equivalent to the non-connective $K$-theory spectrum of Thomason--Trobaugh
\cite{thomason_trobaugh}. The $\infty$-category $\Catperf$ has a symmetric monoidal structure which we will denote by $\widehat{\otimes}$. For $R$ an $\Einf$-ring spectrum, $\Perf(R)$ is an $\Einf$
algebra in $\Catperf$. We will restrict the domain of algebraic $K$-theory to $\Mod_{\Perf (R)}(\Catperf)$. 

\textbf{Acknowledgements:} The author is thankful to his advisor, Benjamin Antieau, for the
suggested project, patience, and guidance. He also thanks Maximilien P\'{e}roux for helpful comments on an earlier
draft. The author was partially supported by NSF Grant DMS-1552766 and NSF RTG grant DMS-1246844.

\section{Twisted $K$-theory}
In Grothendieck's original papers \cite{grothendieck_brauer_i} \cite{grothendieck_brauer_ii}
\cite{grothendieck_brauer_iii}, he globalizes the notion of a central simple algebra over a field.
\begin{defin}
  A locally free sheaf of $\OO_X$-algebras $\azu$ is a \textit{sheaf of Azumaya algebras} if it is \'etale-locally isomorphic to
  $\mathcal{M}_n(\OO_X)$ for some $n$.
\end{defin}
An Azumaya algebra is then a $PGL_n$-torsor over the \'etale topos of $X$ and
so, by Giraud, isomorphism classes are in bijection with $H^1_{\tn{\'et}}(X, PGL_n)$. The central extension of sheaves of groups in the \'etale topology
\[
  1 \rightarrow \mathbb{G}_m \rightarrow GL_n \rightarrow PGL_n \rightarrow 1
\]
leads to an exact sequence of nonabelian cohomology
\[
  \begin{tikzcd}
    \cdots \ar[r] & H_{\tn{\'et}}^1(X,
  \mathbb{G}_m) \ar[r] & H_{\tn{\'et}}^1(X, GL_n) \ar[r] & H_{\tn{\'et}}^1(X, PGL_n) \ar[r] & H_{\tn{\'et}}^2(X,
  \mathbb{G}_m).
  \end{tikzcd}
\]
For $d\,|\,n$ we have a morphism of exact sequences
\[
  \begin{tikzcd}
    1 \ar[r] & \mathbb{G}_m \ar[r] & GL_n \ar[r] & PGL_n \ar[r] & 1 \\
    1 \ar[r] & \mathbb{G}_m \ar[u, equal] \ar[r] & GL_d \ar[u] \ar[r] &  PGL_d \ar[u,
    ] \ar[r] & 1
  \end{tikzcd}
\]
with the two right arrows given by block-summing the matrix along the diagonal $n/d$ times. The
Brauer group is the filtered colimit of cofibers
\[
  Br(X) := \colim(\cofib(H_{\tn{\'et}}^1(X, GL_n) \rightarrow H_{\tn{\'et}}^1(X, PGL_n)))
\]
along the partially-ordered set of the natural numbers under division. This is the group of Azumaya algebras modulo Morita equivalence with group operation given by tensor
product (see \cite{grothendieck_brauer_i}). We have an injection $Br(X) \inj H^2_{\tn{\'et}}(X, \mathbb{G}_m)$ and when $X$ is
quasi-compact this injection
factors through the torsion subgroup. We will call $Br'(X) := H^2_{\tn{\'et}}(X, \mathbb{G}_m)_{tor}$ the cohomological
Brauer group. Grothendieck asked if the injection $Br(X) \inj Br'(X)$ is an
isomorphism.

This map is not generally surjective. Edidin--Hassett--Kresch--Vistoli
\cite{edidin_hassett_kresch_vistoli} give a non-separated counter
example by connecting the image of the Brauer group
to quotient stacks. There are two ways to proceed in addressing the question. The first is to provide a class of schemes
for when this holds. In \cite{deJong_gabber}, de Jong publishes a proof of O. Gabber that $Br(X) \cong Br'(X)$ when
$X$ is equipped with an ample line bundle. Along with reproving Gabber's result for affines, Lieblich \cite{lieblich_thesis} shows
that for a regular scheme with dimension less than or equal to $2$ there are isomorphisms $Br(X) \cong Br'(X)
\cong H^2_{\tn{\'et}}(X, \mathbb{G}_m)$.

The second perspective is to enlarge the class of objects considered. The Morita equivalence classes of
$\mathbb{G}_m$-gerbes over the \'etale topos of a scheme $X$ are in bijection with $H^2_{\tn{\'et}}(X,
\mathbb{G}_m)$. In \cite{lieblich_thesis},
Lieblich associates to any Azumaya algebra $\azu$ a $\mathbb{G}_m$-gerbe of Morita-theoretic trivializations. Over an
\'etale open $U \rightarrow X$, the gerbe gives a groupoid of Morita equivalences from
$\azu$ to $\OO_X$. The gerbe of trivializations represents the boundary class
$\delta([\azu]) = \alpha \in H^2_{\tn{\'et}}(X, \mathbb{G}_m)$.

Any class $\alpha \in H^2_{\tn{\'et}}(X, \mathbb{G}_m)$ is realizable 
on a \v{C}ech cover. We can use this data to build a well-defined category of sheaves of
$\OO_X$-modules which ``glue up to
$\alpha$'', see C\u ald\u araru \cite[Chapter 1]{caldararu_thesis}. Let $\Mod_X^\alpha$ denote
the corresponding derived $\infty$-category and $\Perf_X^\alpha$ the full subcategory of compact objects.
$K(\Perf_X^{\alpha})$ is the classical definition of $\alpha$-twisted
algebraic $K$-theory. Determining when the cohomology class $\alpha$ 
is represented by an Azumaya algebra reduces to finding a twisted locally-free sheaf with trivial
determinant on a 
$\mathbb{G}_m$-gerbe associated to $\alpha$ \cite[Section 2.2.2]{lieblich_thesis}.  The endomorphism algebra of the twisted locally-free sheaf gives the Azumaya
algebra and the twisted module represents the tilt $\Mod_X^{\alpha} \simeq \Mod_\azu$. 

Lieblich also compactifies the moduli of Azumaya algebras. This necessarily includes developing a definition of a  derived Azumaya algebra.
\begin{defin}
  A \textit{derived Azumaya algebra} over a commutative ring $R$ is a proper dg-algebra $\azu$ such
  that the natural map of $R$-algebras
  \[
    \azu \otimes_R^\mathbb{L} \azu^{op} \xrightarrow{\simeq} \mathbb{R}Hom_{D(R)}(\azu, \azu)
  \]
  is a quasi-isomorphism.
\end{defin}
After Lieblich, To\"en \cite{toen_azumaya} and (later)
Antieau--Gepner \cite{antieau_gepner} consider the analogous problem posed by Grothendieck in
the dg-algebra and $\Einf$-algebra settings, respectively. Antieau--Gepner construct an \'etale sheaf
$\mathbf{Br}$ in the $\infty$-topos $\mathrm{Shv}^{\tn{\'et}}_R$. For any \'etale sheaf $X$, we can now
associate a Brauer space $\mathbf{Br}(X)$. For $X$ a quasi-compact quasi-separated scheme, they show
$\pi_0(\mathbf{Br}(X)) \cong H^1_{\tn{\'et}}(X, \mathbb{Z}) \times H^2_{\tn{\'et}}(X, \mathbb{G}_m)$
and every such Brauer class is algebraic. Now for any (possibly nontorsion) $\alpha \in H^2_{\acute{e}t}(X,
\mathbb{G}_m)$ there is a derived Azumaya algebra $\azu$ and an equivalence
$\Mod_X^\alpha \simeq \Mod_\azu$ of stable $\infty$-categories.

This reframes classical twisted $K$-theory as $K$-theory with coefficients in a particularly special dg-algebra in $D(X)$. For our purposes, we work
with a generalized definition of twisted $K$-theory which allows ``twisting'' by any dg-algebra.

\begin{defin}
Let $R$ be a commutative ring. For a dg-algebra $\azu$ over $R$, we define the $\azu$-\textit{twisted
  $K$-theory} $K^\azu: \Mod_{\Perf(R)}(\Catperf) \rightarrow \tn{Sp}$ by $K^\azu(\mathcal{C}) := K( \mathcal{C} \widehat{\otimes}_{\Perf(R)} \Perf(\azu))$.
\end{defin}

When the dg-algebra ``$\azu$'' is clear, we just write twisted K-theory. If our input
to $K^\azu$ is an $R$-algebra $S$ then
\[
  K^\azu(S) = K(\Perf(S) \widehat{\otimes}_{\Perf(R)} \Perf(\azu))
  \hmtpy K(\Perf(S\otimes_R \azu)) \hmtpy K(S \otimes_R \azu).
\]
Our definition recovers the historical definition of twisted K-theory when $\azu$ is a derived
Azumaya algebra and we evaluate on the base ring $R$.  The same definition works for a scheme $X$
and $\azu \in \tn{Alg}_{\Eone}(D_{qc}(X))$. We will refer to such an $\azu$ as \textit{a sheaf of
  quasi-coherent dg-algebras over $X$}. By Theorem 9.36 of Blumberg--Gepner--Tabuada
\cite{blumberg_gepner_tabuada_universal}, twisted $K$-theory is a localizing invariant. When $X$ is
a quasi-compact quasi-separated scheme,
Proposition A.15 of Clausen--Mathew--Naumann--Noel \cite{clausen_mathew_naumann_noel} establishes
Nisnevich descent when $X$ is quasi-compact quasi-separated.
\begin{defin}
  A dg-algebra $\azu$ over a ring $R$ is \textit{proper} if it is perfect as a complex over $R$ and \textit{smooth} if
  it is perfect over $\azu^{op} \otimes_{R} \azu$.
\end{defin}
The following is Lemma 2.8 of \cite{toen_vaquie} and is an essential property for our proof in Section 3.
\begin{lemma}\label{perfection_transference}
  Let $\azu$ be a smooth proper dg-algebra over a ring $R$. Then a complex of $D(\azu)$ is perfect over
  $\azu$ if and only if it is perfect as an object of $D(R)$.
\end{lemma}
The previous definition and lemma both generalize to a sheaf of quasi-coherent dg-algebras over a scheme as perfection is a
local property. For the remainder of the section, we prove some basic properties of $\azu$-twisted K-theory, often
assuming $\azu$ is connective. We will not use smooth and properness until the later sections.
\begin{proposition}
  \label{pi_zero}
  Let $\azu$, $S$ be connective dg-algebras over $R$. Then the natural maps induce isomorphisms
  \[
    K_i^\azu(S) \cong K_i^\azu(\pi_0(S)) \cong K_i^{\pi_0(\azu)}(S) \cong K_i^{\pi_0(\azu)}(\pi_0(S))
  \]
  for $i \leq 0$.
\end{proposition}
\begin{proof}
  We have the following isomorphisms of discrete rings
  \[
    \pi_0(\azu
    \otimes_R S) \cong \pi_0(\azu \otimes_R \pi_0(S)) \cong \pi_0(\pi_0(\azu) \otimes_R S) \cong
    \pi_0(\pi_0(\azu) \otimes_R \pi_0(S)).
    \]
    The lemma follows since $K_i(R) \cong K_i(\pi_0(R))$ for $i \leq 0$ (see Theorem
  9.53 of \cite{blumberg_gepner_tabuada_universal}).
\end{proof}
The previous proposition suggests we can work discretely and then transfer the results to the derived
setting. This is true to some extent. However, taking $\pi_0$ of a connective dg-algebra does not
preserve smoothness, which is a necessary property for our proof of Proposition
\ref{platification}. We will also need reduction invariance for low dimensional K-groups.

    \begin{proposition}\label{reduction_invariance}
      Let $R$ be a commutative ring and $\azu$ a connective dg-algebra over $R$. Let $S$ be a
      commutative ring
      under $R$ and Let $I$ be a
      nilpotent ideal of $S$. Then the induced morphism $K_i^\azu(S) \xrightarrow{\cong}
      K_i^\azu(S/I)$ is an isomorphism for $i \leq 0$.
    \end{proposition}
    \begin{proof}
      By naturality of the fundamental exact sequence of twisted $K$-theory (see (\ref{fundamental}) and
      the surrounding discussion at the
      beginning of Section 3), we can restrict the proof
      to $K_0^\azu$. By Proposition \ref{pi_zero}, we can assume $\azu$ is a discrete algebra. Let $\varphi: S \surj S/I$ be the surjection. After $- \otimes_R \azu$ we have a
      surjection $(\ker \varphi)  \otimes_R \azu \surj \ker(\varphi \otimes_R \azu)$. The nonunital
      ring $(\ker
      \varphi) \otimes_R \azu$ is nilpotent. So $\ker(\varphi \otimes_R
      \azu)$ is nilpotent as well. The proposition follows from nil-invariance of $K_0$.
    \end{proof}
    A Zariski descent spectral sequence argument gives us a global result.
    \begin{corollary}\label{scheme_reduction_invariance}
      Let $X$ be a quasi-compact quasi-separated scheme of finite Krull dimension $d$ and $\azu$ a
      sheaf of connective quasi-coherent dg-algebras over $X$. The natural morphism $f: X_{red} \rightarrow X$ induces isomorphisms
      \[
        K_{-i}^{f^*\azu}(X_{red}) \cong K_{-i}^\azu(X)
      \]
      for $i \geq d$.
    \end{corollary}
    \begin{proof}
      We have descent spectral sequences
       \begin{align*}
        E_2^{p,q} &= H^p_{Zar}(X, (\pi_qK^\azu)^\sim) \Rightarrow \pi_{q-p}K^\azu(X) \tn{ and } \\
        E_2^{p,q} &= H^p_{Zar}(X, f_*(\pi_qK^{f^*(\azu)})^\sim) \Rightarrow
                    \pi_{q-p}K^{f^*\azu}(X_{red}) 
       \end{align*}
       both with differential $d_2 = (2, 1)$. We let $F^\sim$ denote the Zariski sheafification of
       the presheaf $F$. The spectral sequences agree for $q \leq 0$. By Corollary 3.27 of
       \cite{clausen_mathew}, the spectral sequences vanishes for $p > d$.
     \end{proof}
     In Theorem \ref{relative_main_theorem}, we extend our main theorem across smooth affine
     morphisms. We will need reduction invariance in this setting.
     \begin{defin}
       For $f: S \rightarrow X$ a morphism of quasi-compact quasi-separated schemes and $\azu$ a
       sheaf of quasi-coherent dg-algebras over $X$, the \textit{relative $\azu$-twisted $K$-theory of $f$} is
       \[
         K^\azu(f) := \fib(K^\azu(X) \xrightarrow{f^*} K^\azu(S)).
       \] 
     \end{defin}
     As defined, $K^\azu(f)$ is a spectrum. There is an associated presheaf of spectra on the base
     scheme $X$ given by $U \mapsto K^\azu(f_{|_U})$. This presheaf sits in a fiber sequence
     \[
       K^\azu(f) \rightarrow K^\azu \rightarrow K^\azu_S
     \]
     where the presheaf $K^\azu_S$ is also defined by pullback along $f$. Both presheaves $K^\azu$ and $K^\azu_S$
     satisfy Nisnevich descent and so $K^\azu(f)$ does as well.
     \begin{corollary}
       \label{relative_reduction_invariance}
       Let $f: S \rightarrow X$ be an affine morphism of quasi-compact quasi-separated
       schemes. Suppose $X$ has Krull dimension $d$ and let $\azu$ be a sheaf of connective quasi-coherent dg-algebras over $X$. Then
       the commutative diagram
       \[
         \begin{tikzcd}
           S_{red} \ar[d] \ar[r, "f_{red}"] & X_{red} \ar[d, "g"] \\
           S \ar[r, "f"] & X
         \end{tikzcd}
       \]
       induces an isomorphism of relative twisted $K$-theory groups
       \[
         K_{-i}^{g^*\azu}(f_{red})\cong K_{-i}^{\azu}(f)
       \]
       for $i \geq d + 1$.
     \end{corollary}
     \begin{proof}
       We have two descent spectral sequences
       \begin{align*}
        E_2^{p,q} &= H^p_{Zar}(X, (\pi_qK^\azu(f))^\sim) \Rightarrow \pi_{q-p}K^\azu(f)(X) \tn{ and } \\
        E_2^{p,q} &= H^p_{Zar}(X, g_*(\pi_qK^{g^*\azu}(f_{red}))^\sim) \Rightarrow
                    \pi_{q-p}K^{g^*\azu}(f_{red})(X_{red}) \\
       \end{align*}
       with differential of degree $d = (2, 1)$ and $F^\sim$ the sheafification of the presheaf $F$. For an open affine $\Spec R \rightarrow
       X$ with pullback $\Spec A \rightarrow S$ we examine the morphism of long exact sequences when $q \leq 0$
       \[
         \begin{tikzcd}[column sep=tiny]
           \cdots \ar[r] & \pi_{q}K^\azu(R) \ar[d, "\cong"] \ar[r] &
           \pi_{q}K^\azu(A) \ar[d, "\cong"] \ar[r] & \pi_{q-1}K^\azu(f) \ar[d] \ar[r] &\pi_{q-1}K^\azu(R) \ar[d, "\cong"] \ar[r] &
           \pi_{q-1}K^\azu(A) \ar[d, "\cong"] \ar[r] & \cdots \\
           \cdots \ar[r] & \pi_{q}K^\azu(R_{red}) \ar[r] &
           \pi_{q}K^\azu(A_{red}) \ar[r] & \pi_{q-1}K^\azu(f_{red}) \ar[r] &\pi_{q-1}K^\azu(R_{red}) \ar[r] &
           \pi_{q-1}K^\azu(A_{red}) \ar[r] &\cdots
         \end{tikzcd}
       \]
       By the 5-lemma, this induces sheaf isomorphisms $g_*(\pi_qK^{g^*\azu}(f_{red}))^\sim \cong
       (\pi_qK^\azu(f))^\sim$ for $q < 0$ and, as in Corollary \ref{scheme_reduction_invariance}, cohomology vanishes for $p > d$.
     \end{proof}

We will need pro-excision for abstract blow-up squares. Recall that
an abstract blow-up square is a pullback square
      \begin{equation} \label{abs}
        \begin{tikzcd}
        D \ar[d] \ar[r] & \tilde{X} \ar[d] \\
        Y \ar[r] & X          
      \end{tikzcd}
      \tag{$*$}
      \end{equation}
      with $Y \rightarrow X$ a closed
      immersion and $\tilde{X} \rightarrow X$ a proper morphism which restricts to an isomorphism of
       open subschemes
      $\tilde{X} \setminus D \rightarrow X \setminus Y$. The theorem is stated using the
      $\infty$-category of pro-spectra $\cat{Pro}(\Spectra)$, where an object is a small cofiltered
      diagram, $E: \Lambda \rightarrow \Spectra$, valued in spectra. We write $\{E_n\}$
      for the corresponding pro-spectrum. If the brackets
      and index are omitted, then the pro-spectrum is considered constant. After adjusting
      equivalence class representatives, we may assume the
      cofiltered diagram is fixed when working with a finite set of pro-spectra. Any
      morphism can then be represented by a natural transformation of diagrams (also known as a
      level map). We will need no
      knowledge of the $\infty$-category beyond the following definition.
      \begin{defin}
        A square of pro-spectra
        \[
          \begin{tikzcd}
            \{E_n\} \ar[r] \ar[d] & \{F_n\} \ar[d] \ar[d] \\
            \{ X_n\} \ar[r] & \{ Y_n\}
          \end{tikzcd}
        \]
        is \textit{pro-cartesian} if and only if the induced map on the level-wise fiber
        pro-spectra is a weak equivalence (see Definition 2.27 of \cite{land_tamme}).
      \end{defin}
      The following is Theorem A.8
of Land--Tamme \cite{land_tamme}. The theorem holds much more generally for any $k$-connective
localizing invariant (see Definition 2.5 of \cite{land_tamme}). Twisted $K$-theory is
$1$-connective.
\begin{theorem}[Land--Tamme \cite{land_tamme}]\label{land_tamme}
      Given an abstract blow-up square (\ref{abs})
      of schemes and a sheaf of dg-algebras $\azu$ on $X$ then the square of pro-spectra
      \[
        \begin{tikzcd}
        \ar[d] K^\azu(X) \ar[r] & K^\azu(\tilde{X}) \ar[d] \\
        \{K^\azu(Y_n)\} \ar[r] & \{K^\azu(D_n)\}
        \end{tikzcd}
      \]
      is pro-cartesian (where $Y_n$ is the infinitesimal thickening of $Y$).
    \end{theorem}
    The pro-cartesian square of pro-spectra gives a long exact sequence of pro-groups
    \[
      \begin{tikzcd}[column sep=small]
            \cdots \ar[r] & \{K_{-i+1}^\azu(E_n)\} \ar[r]  & K_{-i}^\azu(X) \ar[r] & K_{-i}^\azu(\tilde{X}) \oplus \{ K_{-i}^\azu(Y_n)\} \ar[r]& \{ K_{-i}^\azu(E_n)\} \ar[r] & \cdots
      \end{tikzcd}
    \]
    which is the key to our induction argument.

      \section{Blowing-up negative twisted $K$-theory classes}
      We turn to our main contribution of the existence of a projective
      birational morphism which kills a given negative twisted $K$-theory class (when twisting by a smooth
      proper connective dg-algebra). Let $X$ be a quasi-compact quasi-separated scheme and $\azu$ a
      sheaf of quasi-coherent dg-algebras on $X$. We first construct geometric cycles for negative twisted
      K-theory classes on $X$ using a classical argument of Bass (see XII.7 of
      \cite{bass_algebraic_k-theory}) which works for a general additive invariant. We have an open cover
      \[
        \begin{tikzcd}
          X[t^{\pm}] \ar[r, "f"] \ar[d, "g"] & X[t^-] \ar[d, "j"] \\
          X[t] \ar[r, "k"] & \proj^1_X.
        \end{tikzcd}
      \]
      Since twisted $K$-theory satisfies Zariski descent, there is an associated Mayer-Vietoris sequence of
      homotopy groups
      \[
        \begin{tikzcd}
          \cdots \ar[r] &K_{-n}^\azu(\proj^1_X) \ar[r, "{(j^* k^*)}"] & K_{-n}^\azu(X[t]) \oplus K_{-n}^\azu(X[t^-]) \ar[r, "f^*-g^*"]
          & K^\azu_{-n}(X[t^{\pm}]) \ar[r, "\partial"] & K^\azu_{-n-1}(\proj^1_X) \ar[r] & \cdots
        \end{tikzcd}.
      \]
      As an additive invariant, $K^\azu(\proj^1_X) \simeq K^\azu(X) \oplus K^\azu(X)$ splits as a
      $K^\azu(X)$-module with generators
      \[
        [\OO \otimes_{\OO_X} \azu] = [\azu] \text{ and } [\OO(1)
        \otimes_{\OO_X} \azu]=[\azu(1)]
      \]
      corresponding to the Beilinson semiorthogonal
      decomposition. Adjusting the generators to $[\azu]$ and $[\azu] - [\azu(1)]$, we can identify
      the map $(j^*, k^*)$ as it is a map of $K^\azu(X)$-modules. The second generator vanishes
      under each restriction. This identifies the map as
      \[
        K^\azu(\proj^1_X) \simeq K^\azu(X)[\azu] \oplus K^\azu(X)([\azu] - [\azu(1)])
        \xrightarrow{\Delta \oplus 0} K^\azu(X[t]) \oplus K^\azu(X[t^-])
      \]
      with $\Delta$ the diagonal map corresponding to pulling back along the projections $X[t]
      \rightarrow X$ and $X[t^-] \rightarrow X$. As $\Delta$ is an embedding the long exact sequence
      splits as
      \begin{equation}\label{fundamental}
        \begin{tikzcd}
          0 \ar[r] & K_{-n}^\azu(X) \ar[r, "\Delta"] & K_{-n}^\azu(X[t]) \oplus K_{-n}^\azu(X[t^-]) \ar[r, "\pm"]
          & K^\azu_{-n}(X[t^{\pm}]) \ar[r, "\partial"] & K^\azu_{-n-1}(X)\ar[r] & 0
        \end{tikzcd}.
        \tag{$\dagger$}
      \end{equation}
      After iterating the complex
      \[
        K_{-n}^\azu(X[t]) \rightarrow K^\azu_{-n}(X[t^\pm]) \surj K^\azu_{-n-1}(X),
      \]
      we can piece together a complex
      \[
        K_0^\azu(\aff^{n+1}_X) \rightarrow K^\azu_0(\mathbb{G}_{m,X}^{n+1}) \surj K^\azu_{-n-1}(X).
      \]
      Negative twisted $K$-theory classes have geometric
      representations as twisted perfect complexes on $\mathbb{G}^i_{m, X}$. There is even a
      sufficient geometric criterion implying a given representative
      is $0$; it is the restriction of a twisted perfect complex on $\aff^i_{X}$. Our proof of the main
      proposition of this section will use these representatives. We first need a lemma about extending finitely-generated discrete modules in a twisted
      setting.

      \begin{lemma}
        \label{twisted_extension}
        Let $j: U \rightarrow X$ be an open immersion of quasi-compact quasi-separated schemes. Let $\azu$ be a sheaf
        of proper connective quasi-coherent dg-algebras on $X$ and $j^*\azu$ its
        restriction. Let $\mathcal{N}$ be a discrete $j^*\azu$-module which is finitely generated as an
        $\OO_U$-module. Then there exists a discrete $\azu$-module $\mathcal{M}$, finitely
        generated over $\OO_X$, such that $j^*\mathcal{M} \cong \mathcal{N}$.
      \end{lemma}
      \begin{proof}
        Note that $H_{\geq 1}(j^*\azu)$ necessarily acts trivially on
    $\mathcal{N}$. So the $j^*\azu$-module structure on $\mathcal{N}$ comes from forgetting
    along the map $j^*\azu
    \rightarrow H_0(j^*\azu)$ and the natural $H_0(j^*\azu)$-module structure. Under
    restriction,
    \[
      j^*H_0(\azu) \cong H_0(j^*\azu).
    \]
    We reduce to when $\azu$ is a quasi-coherent sheaf of discrete $\OO_X$-algebras, finite over the
    structure sheaf. We have an isomorphism $\mathcal{N} \cong j^*j_*\mathcal{N}$. Write
    $j_*\mathcal{N}$ as a filtered colimit of its finitely generated $\azu$-submodules
    $j_*\mathcal{N} \cong \underset{\lambda}\colim \mathcal{M}_\lambda$. The pullback is exact, so
    we can write $\mathcal{N} \cong \underset{\lambda}\colim j^*\mathcal{M}_\lambda$ as a filtered
    colimit of finitely generated submodules. As $\mathcal{N}$ is finitely generated itself, this
    isomorphism factors at some stage and $\mathcal{N} \cong j^*\mathcal{M}_\lambda$.
  \end{proof}

\begin{proposition}
  \label{platification}
  Let $X$ be a reduced scheme which is quasi-projective over a Noetherian affine scheme. Let $\azu$ be a
  sheaf of smooth proper
  connective quasi-coherent dg-algebras on $X$. Let $\gamma \in K_{-i}^\azu (X)$
  for $i > 0$. Then there is a projective birational morphism $\rho: \tilde{X} \rightarrow X$ so
  that $\rho^*\gamma = 0 \in K_{-i}^\azu(\tilde{X})$.
\end{proposition}
\begin{proof}
  We fix a diagram of schemes over $X$
  \[
    \begin{tikzcd}
      \mathbb{G}_{m, X}^i \ar[dr, "\pitor"'] \ar[rr, "j"] & & \ar[ld, "\piaff"] \aff_X^i  \\
      & X & 
    \end{tikzcd}.
  \]
  For any morphism $f: Y_1 \rightarrow Y_2$, we let $\tilde{f}: \mathbb{G}_{m, Y_1}^i
  \rightarrow \mathbb{G}_{m, Y_2}^i$ denote the pullback.   Lift $\gamma$ to a
  $K_0^\azu(\mathbb{G}_{m,X}^i)$-class $[P_\bullet]$, with $P_\bullet$ some
  $\pitor^*\azu$-twisted perfect complexes on $\mathbb{G}_{m,X}^i$. \\
  
  \textit{The Induction Step}: \\
  
We induct on the range of homology of $P_\bullet$. As $\pitor^*\azu$
  is a sheaf of proper quasi-coherent dg-algebras, $P_\bullet$ is perfect on $\mathbb{G}_{m,X}^i$ by Lemma \ref{perfection_transference}. Since $\mathbb{G}_{m,X}^i$ has an ample
  family of line bundles, we may choose $P_\bullet$ to be strict perfect without changing the
  quasi-isomorphism class. After some (de)suspension, we may assume $P_\bullet$ is connective
  as this only alters the $K_0$-class by $\pm 1$.  For the lowest nontrivial differential of $P_\bullet$, $d_\low$, we utilize part (iv) of Lemma 6.5
  of \cite{kerz_strunk_tamme} (with the morphism $\mathbb{G}_{m, X}^i \rightarrow X$) to construct a
  projective birational morphism $\rho: X_1 \rightarrow X$ so that $\coker(\tilde{\rho}^*
  d_\low)$ ($=H_0(\tilde{\rho}^*P_\bullet)$) has tor-dimension $\leq 1$ over $X_1$. Consider the
  following distinguished triangle of $\tilde{\rho}^*\pitor^*\azu$-complexes on $\mathbb{G}_{m, X_1}^i$
  \[
    F_\bullet \rightarrow \tilde{\rho}^*P_\bullet \rightarrow H_0(\tilde{\rho}^*P_\bullet) \cong \coker \tilde{\rho}^*d_\low.
  \]
  In Lemma \ref{base_platification} below, we cover the base induction step, when the homology is concentrated in
  a single degree. Using this, construct a projective birational morphism $\phi: X_2\rightarrow X_1$ such that
  $L\tilde{\phi}^*H_0(\tilde{\rho}^*P_\bullet)$ is a perfect complex and is the restriction of a perfect complex
  from $\aff^i_{X_2}$. By two out of three, $L\tilde{\phi}^*F_\bullet$ is perfect and
  $[\tilde{\phi}^*\tilde{\rho^*}P_\bullet] = [L\tilde{\phi}^*F_\bullet] + [L\tilde{\phi}^*H_0(\tilde{\rho}^*P_\bullet)]$ in
  $K_0^\azu(\mathbb{G}_{m, X_2}^i)$. We then repeat the entire induction step with $L\tilde{\phi}^*F_\bullet$.

We need the induction will terminate, which is the purpose of the first projective
birational morphism of each step. Since  $\coker(\tilde{\rho}^*
  d_\low)$ has tor-dimension $\leq 1$ over $X_1$, by \cite{kerz_strunk_tamme}[Lemma 6.5],
  $L\tilde{\phi}^*\coker (\tilde{\rho}^*d_\low) \cong \tilde{\phi}^* \coker
  (\tilde{\rho}^*d_\low)$. This implies $L \tilde{\phi}^*F_\bullet$ will have no homology outside the original range
  of homology
  of $P_\bullet$. Since $\tilde{\phi}^* \coker (\tilde{\rho}^*d_\low) \cong \coker
  (\tilde{\phi}^*\tilde{\rho}^*d_\low)$, this guarantees $H_0(L\tilde{\phi}^*F_\bullet) = 0$, so the
  homology of $L\tilde{\phi}^*F_\bullet$ lies in a strictly smaller range than
  $\tilde{\phi}^*\tilde{\rho}^*P_\bullet$. Proposition \ref{platification} follows from the next
  lemma.
  \end{proof}

  \begin{lemma}\label{base_platification}
    Let $X$ be a reduced scheme which is quasi-projective over a Noetherian affine
    scheme. Let $\azu$ be a sheaf of smooth proper connective quasi-coherent dg-algebras on $X$. Let
    $\sheaffont{N}$ be a discrete $\pitor^*\azu$-module which is coherent on $\mathbb{G}_{m,
      X}^i$. Then there exists a birational blow-up $\phi: \tilde{X} \rightarrow X$ so that
    $\tilde{\phi}^*\sheaffont{N}$ is perfect over $\tilde{\phi}^*\pitor^*\azu$ on $\mathbb{G}_{m,
      \tilde{X}}$ and is the 
    restriction of a perfect complex over the pullback of $\azu$ to $\aff^i_{\tilde{X}}$.
  \end{lemma}
  \begin{proof}
    Using Lemma \ref{twisted_extension}, extend $\sheaffont{N}$ from $\mathbb{G}_{m, X}^i$ to a
    coherent $\piaff^*\azu$-module $\sheaffont{M}$ on $\mathbb{A}_X^i$. Using the ample family, choose a resolution in $\OO_{\aff_X^i}$-modules of the form
  \[
    0 \rightarrow \sheaffont{K} \rightarrow \sheaffont{F} \rightarrow \sheaffont{M} \rightarrow 0
  \]
      where $\sheaffont{F}$ is a vector bundle and $\sheaffont{K}$ is the coherent kernel. As $X$ is
      reduced, $\sheaffont{K}$ is flat over some dense
      open set $U$ of $X$. By platification par \'eclatement (see Theorem 5.2.2 of Raynaud--Gruson \cite{raynaud_gruson}), there is a $U$-admissable
      blow-up $\phi: \tilde{X} \rightarrow X$ so that the strict transform of $\sheaffont{K}$ along the
      pullback morphism $p: \aff_{\tilde{X}}^i \rightarrow \aff_X^i$ is flat over $\tilde{X}$.
      
        We now show the pullback $p^*\sheaffont{M}$ is perfect as a $p^*\piaff^*\azu$-module. Let $j: \aff_U^i \rightarrow \aff_{\tilde{X}}^i$ be the inclusion of the open set and $Z$ the closed complement. For any sheaf of modules
        $\sheaffont{G}$ on $\aff_{\tilde{X}}^i$, we let $\sheaffont{G}_Z$ denote the subsheaf of
        sections supported on $Z$. We have a short exact sequence natural in $\sheaffont{G}$
        \[
          0 \rightarrow \sheaffont{G}_Z \rightarrow \sheaffont{G} \rightarrow j^{st}\sheaffont{G} \rightarrow 0.
        \]
        We also obtain the following exact sequence of sheaves of abelian groups via pullback
        \[
          0 \rightarrow \torsheaf_1^{p^{-1}\OO_{\aff^i_X}}(p^{-1}\sheaffont{M}, \OO_{\aff^i_{\tilde{X}}})
          \rightarrow p^*\sheaffont{K} \rightarrow p^*\sheaffont{F} \rightarrow p^*\sheaffont{M}
          \rightarrow 0.
        \]
        To make our notation clearer, we set $\sheaffont{T} = \torsheaf_1^{
          p^{-1}\OO_{\aff^i_X}}(p^{-1}\sheaffont{M}, \OO_{\aff^i_{\tilde{X}}})$. We flesh both these
        exact sequences out into a (nonexact) commutative diagram of $p^{-1}\OO_{\aff^i_X}$-modules
        \[
          \begin{tikzcd}
            & 0 \ar[d] & 0 \ar[d] & 0 \ar[d] & \\
            0 \ar[r] & \sheaffont{T}_Z
            \ar[r]\ar[d] & \sheaffont{T} \ar[d]\ar[r] & \ar[d] j^{st}\sheaffont{T} \ar[r] & 0 \\
            0 \ar[r] & (p^*\sheaffont{K})_Z \ar[r] \ar[d] & p^*\sheaffont{K} \ar[r] \ar[d] & j^{st}p^*\sheaffont{K}
            \ar[r] \ar[d] & 0\\
            0 \ar[r] & (p^*\sheaffont{F})_Z\ar[r] \ar[d] & p^*\sheaffont{F} \ar[r] \ar[d] &j^{st}p^*\sheaffont{F} \ar[r]
            \ar[d] & 0\\
            0 \ar[r] &(p^*\sheaffont{M})_Z \ar[r] \ar[d] & p^*\sheaffont{M} \ar[r] \ar[d] & j^{st}p^*\sheaffont{M} \ar[r] \ar[d] & 0\\
            & 0 & 0 & 0 &
          \end{tikzcd}.
        \]
        We observe that every row and the middle column is exact. The first map in the left column
        is an injection and the last map in the right column is a surjection. Since $p^*\sheaffont{F}$ is flat, we have $(p^*\sheaffont{F})_Z =
        0$. This induces a lifting of the injection
        \[
          \begin{tikzcd}
            \sheaffont{T}_Z \ar[d] \ar[r] & \sheaffont{T} \ar[d] \\
            (p^*\sheaffont{K})_Z \ar[r] \ar[ur, dashed] & p^*\sheaffont{K}
          \end{tikzcd}.
        \]
        We finish the proof by showing $j^*\torsheaf_1^{p^{-1}\OO_{\aff^i_X}}(p^{-1}\sheaffont{M},
        \OO_{\aff^i_{\tilde{X}}}) = 0$. Since $j: \aff^i_U \rightarrow \aff^i_{\tilde{X}}$ is flat, the sheaf is isomorphic to
        $\torsheaf_1^{\aff^i_U}(j^*p^{-1}\sheaffont{M}, j^*\OO_{\aff^i_{\tilde{X}}})$ and
        $j^*\OO_{\aff^i_{\tilde{X}}} \cong \OO_{\aff^i_U}$.  Our big diagram can be rewritten as
                \[
          \begin{tikzcd}
            & 0 \ar[d] & 0 \ar[d] & 0 \ar[d] & \\
            0 \ar[r] & \sheaffont{T}_Z 
            \ar[r, "\cong"]\ar[d, "\cong"] & \sheaffont{T} \ar[d]\ar[r] & \ar[d] 0 \ar[r] & 0 \\
            0 \ar[r] & (p^*\sheaffont{K})_Z \ar[r] \ar[d] & p^*\sheaffont{K} \ar[r] \ar[d] & j^{st}p^*\sheaffont{K}
            \ar[r] \ar[d] & 0\\
            0 \ar[r] & 0\ar[r] \ar[d] & p^*\sheaffont{F} \ar[r] \ar[d] &j^{st}p^*\sheaffont{F} \ar[r]
            \ar[d] & 0\\
            0 \ar[r] &(p^*\sheaffont{M})_Z \ar[r] \ar[d] & p^*\sheaffont{M} \ar[r] \ar[d] & j^{st}p^*\sheaffont{M} \ar[r] \ar[d] & 0\\
            & 0 & 0 & 0 &
          \end{tikzcd}
        \]
        and we can glue together to get a flat resolution of $p^*\sheaffont{M}$ as an $\OO_{\aff^i_{\tilde{X}}}$-module
        \[
          0 \rightarrow j^{st}p^*\sheaffont{K} \rightarrow p^*\sheaffont{F} \rightarrow p^*\sheaffont{M}
          \rightarrow 0
        \]
        implying globally finite Tor-amplitude. It remains to show the complex is
        pseudo-coherent. This follows since $\aff^i_{\tilde{X}}$ is Noetherian and
        $p^*\sheaffont{M}$ is coherent. Since $p^*\piaff^*\azu$ is a sheaf of smooth quasi-coherent dg-algebras
        over $\OO_{\aff^i_{\tilde{X}}}$, the complex $p^*\sheaffont{M}$ is perfect over
        $p^*\piaff^*\azu$ by Lemma \ref{perfection_transference}. By commutativity, $p^*\sheaffont{M}$ restricts to $\tilde{\phi}^*\sheaffont{N}$ on $\mathbb{G}_{m, \tilde{X}}^i$. This
        completes the proof of Proposition \ref{platification}.
      \end{proof}
    We will need a relative version of Proposition \ref{platification}.
      \begin{corollary}
        \label{relative_platification}
        Let $f: S \rightarrow X$ be a smooth quasi-projective morphism of Noetherian schemes with
        $X$ reduced and  quasi-projective over a Noetherian base ring. Let $\azu$ be a sheaf of smooth proper
        connective quasi-coherent dg-algebras over $X$ and consider a negative twisted $K$-theory class $\gamma \in
        K_{i}^\azu(S)$ for $i < 0$. Then there exists a projective birational morphism $\rho: \tilde{X}
        \rightarrow X$ such that, under the pullback of the pullback morphism, $\rho_S^*\gamma = 0$.
      \end{corollary}
      \begin{proof}
        We will briefly check that we can run the induction argument in the proof of Proposition
        \ref{platification}. The assumptions of this corollary are invariant under pullback along
        projective birational morphisms $\tilde{X} \rightarrow X$. We need to ensure we can select
        projective birational morphisms to our base $X$. Lemma 6.5 of 
        Kerz--Strunk--Tamme \cite{kerz_strunk_tamme} is stated in a relative setting. The proof also
        relies on platification par \'{e}clatement. This can still be applied in our relative
        setting as $X$ is reduced (see Proposition 5 of Kerz--Strunk \cite{kerz_strunk}).
      \end{proof}

      \section{Twisted Weibel's conjecture}

        We now prove Theorem \ref{main_theorem} and an extension across a smooth affine morphism. We
      begin with the base induction step for both theorems. Kerz--Strunk \cite{kerz_strunk} use a
      sheaf cohomology result of Grothendieck along with a spectral sequence argument to show
      vanishing for a Zariski sheaf of spectra can be reduced to the setting of local ring.
      \begin{proposition}\label{base_case}
        Let $R$ be a regular Noetherian ring of Krull dimension $d$ over a local Artinian ring $k$. Let $\azu$
        be a smooth proper connective dg-algebra over $R$, then $K^\azu_i(R) = 0$ for $i < 0$. 
      \end{proposition}
      \begin{proof}
        By Proposition \ref{relative_reduction_invariance}, we may assume $k$ is a field. Proposition 5.4 of \cite{raedschelders_stevenson} shows that the t-structure on $D(\azu)$
        restricts to a t-structure on $\Perf(\azu)$, which is observably bounded. The heart is the
        category of finitely-generated modules over $
        H_0(\azu)$. As $H_0(\azu)$ is finite-dimensional over $k$, this is a Noetherian abelian
        category. By Theorem 1.2 of Antieau--Gepner--Heller
        \cite{antieau_gepner_heller}), the negative $K$-theory vanishes.
      \end{proof}
      \begin{customthm}{1.1}
        Let $X$ be a Noetherian scheme of Krull dimension $d$ and
        $\azu$ a sheaf of smooth proper connective quasi-coherent dg-algebras on $X$, then $K^\azu_{-i}(X)$
        vanishes for $i > d$.
      \end{customthm}
      \begin{proof}
    Proposition \ref{base_case} covers the base case so assume $d > 0$. By the Kerz--Strunk spectral
    sequence argument and Corollary
    \ref{scheme_reduction_invariance}, we may assume $X$ is a Noetherian reduced affine scheme.

    Choose a negative $K^\azu$-theory class $\gamma \in K^\azu_{-i}(X)$ for $i \geq \dim X +
    1$. Using Proposition \ref{platification}, construct a projective birational morphism that kills
    $\gamma$ and extend it to an abstract blow-up square
\[
  \begin{tikzcd}
    E \ar[d] \ar[r] & \tilde{X} \ar[d] \\
    Y \ar[r] & X
  \end{tikzcd}.
\]
By \cite[Theorem A.8]{land_tamme}, there is a Mayer-Vietoris exact sequence of pro-groups
\[
  \begin{tikzcd}[column sep=small]
    \cdots \ar[r] & \{ K_{-i+1}^\azu(E_n)\} \ar[r]  & K_{-i}^\azu(X) \ar[r] & K_{-i}^\azu(\tilde{X}) \oplus \{ K_{-i}^\azu(Y_n)\} \ar[r]& \{ K_{-i}^\azu(E_n)\} \ar[r] & \cdots
  \end{tikzcd}.
\]
When $i \geq \dim X + 1$, by induction every nonconstant pro-group vanishes and $K_{-i}^\azu(X) \cong K_{-i}^\azu(\tilde{X})$ showing $\gamma = 0$.
\end{proof}
By \cite[Theorem 3.15]{antieau_gepner}, we recover Weibel's vanishing for discrete Azumaya algebras.
\begin{corollary}
  For $X$ a Noetherian $d$-dimensional scheme and $\azu$ a quasi-coherent sheaf of discrete
  Azumaya algebras, then $K_{-i}^\azu(X) = 0$ for $i > d$.
\end{corollary}
The next result nearly covers the K-regularity portion of Weibel's conjecture, but we are missing the
boundary case $K_{-d}^\azu(X) \cong K_{-d}^\azu(\aff^n_X)$.
\begin{theorem}
  \label{relative_main_theorem}
  Let $f: S \rightarrow X$ be a smooth affine morphism of Noetherian schemes and $\azu$ a sheaf of smooth
  proper connective quasi-coherent dg-algebras on $X$. Then $K_{-i}^\azu(f) = 0$ for $i > \dim X + 1$.
\end{theorem}
\begin{proof}
  The base case is covered by Proposition \ref{base_case} and our reductions are analagous to those
  in the proof of Theorem \ref{main_theorem}. So assume $X$ is a Noetherian reduced affine scheme of
  dimension $d$. Choose $\gamma \in
K_{-i}^\azu(S)$ with $i > d$. Using Corollary \ref{relative_platification}, construct a projective
birational morphism $\rho: \tilde{X} \rightarrow X$ that kills $\gamma$. We then build a morphism of
abstract blow-up squares
\[
  \begin{tikzcd}[column sep=small, row sep=small]
    D \ar[dd] \ar[rr] \ar[rd] & & \tilde{S} \ar[dd]\ar[rd] & \\
    & E \ar[rr] \ar[dd] & & \tilde{X} \ar[dd] \\
    V \ar[rr] \ar[rd] & & S \ar[rd] &\\
    & Y \ar[rr] & & X
  \end{tikzcd}
\]
By Theorem \ref{land_tamme}, we again get a long exact sequence of pro-groups corresponding to the back square
\[
    \begin{tikzcd}[column sep=small]
    \cdots \ar[r] & \{ K_{-i+1}^\azu(D_n)\} \ar[r]  & K_{-i}^\azu(S) \ar[r] & K_{-i}^\azu(\tilde{S}) \oplus \{ K_{-i}^\azu(V_n)\} \ar[r] & \{ K_{-i}^\azu(D_n)\} \ar[r] & \cdots
  \end{tikzcd}.
\]
When $i \geq \dim X +1$, every nonconstant pro-group vanishes by induction and we have an isomorphism $K_{-i}^\azu(S) \cong K_{-i}^\azu(\tilde{S})$ implying $\gamma = 0$.
\end{proof}
  \begin{rem} \label{counter_example}  The conditions on the morphism in Corollary \ref{relative_platification} are more
    general than those of Theorem \ref{relative_main_theorem}. We might hope to generalize
    Theorem \ref{relative_main_theorem} to a smooth quasi-projective or smooth projective map of Noetherian
    schemes. Although the induction
    step is present, both base cases fail. Consider the descent spectral sequence
\[
  E_2^{p, q} := H^p(X, \tilde{K_{q}}) \Rightarrow K_{q-p}(X) \text{ with }d_2 = (2, 1)
\]
If $\dim X \leq 3$, then
\[
  E_3^{2,1} = E_{\infty}^{2, 1} = \coker(H^0(X, \ZZ) \xrightarrow{d_2} H^2(X, \OO_X^*))
\]
contributes to $K_{-1}(X)$. The differential is zero as the edge morphism
\[
  \begin{tikzcd}
    K_0(X) \ar[r, two heads, "rank"] &     E_\infty^{0, 0}
  \end{tikzcd}
\]
identifies $E_\infty^{0,0}$ with the rank component of $K_0$, implying $E_2^{0,0} = E_\infty^{0,
  0}$. We now construct a family of examples for schemes $X$ with nontrivial $H^2(X, \OO_X^*)$. Let $X_{red}$ be quasi-projective smooth over a field $k$ and form the
cartesian diagram
\[
  \begin{tikzcd}
     X \ar[r, "f"] \ar[d] & X_{red} \ar[d] \\
    \Spec (k[t]/(t^2)) \ar[r] & \Spec k
  \end{tikzcd}.
\]
The pullback $X$ will be our counter-example. We have an isomorphism
\[
  \OO_{X}^* \cong g_*(\OO_{X_{red}}^*) \oplus g_*(\OO_{X_{red}})
\]
of sheaves of abelian groups on $X$ with $g: X_{red} \rightarrow X$ the pullback of the reduction morphism
$\Spec k \rightarrow \Spec k[t]/(t^2)$. Locally, $(R[t]/(t^2))^\times$ consists of all elements of the form $u + v\cdot t$ where $u \in R^\times$ and $v \in R$. Sheaf cohomology commutes with coproducts so this turns into an isomorphism
\[
  H^2(X, \OO_X^*) \cong H^2(X, g_*(\OO_{X_{red}}^*)) \oplus H^2(X, g_*(\OO_{X_{red}})) \cong
  H^2(X_{red}, \OO_{X_{red}}^*) \oplus H^2(X_{red}, \OO_{X_{red}}).
\]
Now the problem reduces to finding a surface or $3$-fold $X_{red}$ with nontrivial degree $2$ sheaf
cohomology. Take a smooth quartic in $\proj^3_k$ for a counter-example which is smooth and proper. Here
is a counter-example which is smooth and quasi-affine. Let $(A, \mm)$ be a 3-dimensional local ring
which is smooth over a field $k$. Take $X = \Spec A
\setminus \{\mm\}$ to be the punctured spectrum. Then $H^2(X, \OO_X) \cong H^3_\mm(A)$, which is the
injective hull of the residue field $A/\mm$.
\end{rem}


 \end{document}